\documentclass[11pt, reqno]{amsart}
\usepackage{amsmath,amsopn,amssymb,amsthm,multicol}
\usepackage[cp1251]{inputenc}
\usepackage[english]{babel}
\usepackage[breaklinks=true,colorlinks=true,linkcolor=blue,citecolor=blue,urlcolor=blue]{hyperref}
\usepackage[dvipdf]{graphicx}
\usepackage{mathrsfs}
\usepackage{epstopdf}

\renewenvironment{proof}[1][Proof]{\textbf{#1.} }
{\ \rule{0.5em}{0.5em}}

\renewcommand{\arraystretch}{1.5}

\newtheorem{theorem}{Theorem}

\newtheorem{remark}{Remark}

\newtheorem{example}{Example}

\begin{document}

\title[Refined behavior  ... \dots]
{Refined behavior description of the normalized Ricci flow on homogeneous spaces}

\author{Nurlan A. Abiev}
\address{N.\,A.~Abiev \newline
Institute of Mathematics NAS of the Kyrgyz Republic, Bishkek, prospect Chui, 265a, 720071, Kyrgyz Republic}
\email{abievn@mail.ru}

\begin{abstract}
This article deals with the problems
of preserving the Ricci curvature positivity on homogeneous spaces under
the normalized Ricci flow (NRF).
We found out infinitely many generalized Wallach spaces (GWS) on which
the positivity of the Ricci curvature of metrics
is preserved when evolved by the NRF. Analogously, the number of  GWS is infinite as well, when the positivity of the Ricci curvature can be lost.
We also obtain some refinements to our previous results
devoted to the case of coincided parameters. 
A series of examples is discussed.

\vspace{2mm} \noindent Key words and phrases:
Homogeneous space,  generalized Wallach space,
Riemannian metric, normalized Ricci flow,  Ricci curvature.
\vspace{2mm}

\end{abstract}

\maketitle

\section*{Introduction}

In~\cite{AANS1} we initiated the study of the normalized Ricci flow (NRF) equation
\begin{equation}\label{ricciflow}
\dfrac {\partial}{\partial t} \bold{g}(t) = -2 \operatorname{Ric}_{\bold{g}}+ 2{\bold{g}(t)}\frac{S_{\bold{g}}}{n}
\end{equation}
on  generalized Wallach spaces (GWS), where $\operatorname{Ric}_{\bold{g}}$ and $S_{\bold{g}}$  mean respectively the Ricci tensor and  the scalar curvature of a one-parameter family of Riemannian metrics~$\bold{g}(t)$.
A GWS is a~homogeneous space $G/H$ with a compact Lie group~$G$ and its closed Lie subgroup~$H$ such that $\mathfrak{p}$ in $\mathfrak{g}=\mathfrak{h}\oplus \mathfrak{p}$ admits a  decomposition
$\mathfrak{p}=\mathfrak{p}_1\oplus \mathfrak{p}_2\oplus \mathfrak{p}_3$
into a direct sum of three  irreducible modules
$\mathfrak{p}_1$, $\mathfrak{p}_2$ and $\mathfrak{p}_3$
satisfying $[\mathfrak{p}_i,\mathfrak{p}_i]\subset \mathfrak{h}$ for each $i\in\{1,2,3\}$, where~$\mathfrak{p}$ is an orthogonal complement to the Lie algebra~$\mathfrak{h}$ of $H$ in the Lie algebra~$\mathfrak{g}$ of $G$ with respect to the inner product $\langle\boldsymbol{\cdot}
\,,\boldsymbol{\cdot}\rangle=-B(\boldsymbol{\cdot}\,,\boldsymbol{\cdot})$
 defined on~$\mathfrak{g}$ by the Killing form
$B(\boldsymbol{\cdot}\,,\boldsymbol{\cdot})$ of~$\mathfrak{g}$
(see~\cite{Nikonorov4} and references therein for more details).
Note  that every GWS is described by a triple of real numbers $(a_1,a_2,a_3)\in (0,1/2]^3$. For our purposes within the framework of this article, this is quite sufficient.
The second useful feature of GWS is that any invariant Riemannian
metric~${\bf g}$ and its Ricci tensor $\operatorname{Ric}_{\bf g}$ admit the following decompositions  (see~\cite{Nikonorov4}):
\begin{eqnarray}\label{metric}
{\bf g}(\cdot, \cdot)&=&
\left.x_1\langle\cdot,\cdot\rangle\right|_{\mathfrak{p}_1}+
\left.x_2\langle\cdot,\cdot\rangle\right|_{\mathfrak{p}_2}+
\left.x_3\langle\cdot,\cdot\rangle\right|_{\mathfrak{p}_3},
\\ \label{tensor}
\operatorname{Ric}_{\bf g}(\cdot, \cdot)&=&
\left.x_1 {\bf r}_1\langle\cdot, \cdot\rangle\right|_{\mathfrak{p}_1}+
\left.x_2 {\bf r}_2\langle\cdot, \cdot\rangle\right|_{\mathfrak{p}_2}+
\left.x_3 {\bf r}_3\langle\cdot, \cdot\rangle\right|_{\mathfrak{p}_3},
\end{eqnarray}
where $x_1,x_2,x_3$ are positive real numbers and  ${\bf r}_i={\bf r}_i(x_1,x_2,x_3)\in \mathbb{R}$ are components of the Ricci tensor.
Expressions~\eqref{metric} and \eqref{tensor} allow  to split~\eqref{ricciflow} equivalently into the following system of nonlinear autonomous ordinary differential equations
on a given GWS with respect to $x_1, x_2$ and $x_3$:
\begin{equation}\label{three_equat}
\dot{x}_i= -\,2x_i \left(\bold{r}_i-\frac{a_1^{-1}\,\bold{r}_1+a_2^{-1}\,\bold{r}_2
+a_3^{-1}\,\bold{r}_3}
{a_1^{-1}+a_2^{-1}+a_3^{-1}}\right), \quad i=1,2,3.
\end{equation}
The following formulas are well-known for the components ${\bf r}_i$ of the Ricci tensor on a given GWS:
\begin{equation}\label{Ricci_princ}
{\bf r}_i = \dfrac{1}{2x_i}+\dfrac{a_i}{2}
\left(\dfrac{x_i}{x_j x_k}-\dfrac{x_j}{x_kx_i}-\dfrac{x_k}{x_ix_j} \right),
\end{equation}
where $(i,j,k)$ should be chosen equal to $(1,2,3)$, $(2,3,1)$ and $(3,1,2)$ (see~\cite{AANS1} for details).

One of the important questions is if the NRF preserves the positivity of the Ricci curvature of Riemannian metrics on a given manifold.
In general, the answer to this question turned out to be negative: the Ricci flow need not to preserve the positivity of  curvatures of metrics.
In 2004, Ni showed in~\cite{Ni2}  non-compact Riemannian manifolds with bounded non-negative sectional curvature on which the Ricci flow does not preserve the non-negativity of sectional curvature.
The positivity of curvatures may not be preserved on compact manifolds either,
as it became clear in 2007  by
B\"ohm and Wilking proved in~\cite[Theorem 3.1]{Bo} the following statement:
``On the compact manifold $\operatorname{Sp}(3)/\operatorname{Sp}(1)\times \operatorname{Sp}(1)\times \operatorname{Sp}(1)$  the Ricci flow evolves  \textit{certain}\/ positively curved metrics into metrics with mixed Ricci curvature.''

In~2015, Cheung and Wallach proved in~\cite[Theorem 2]{ChWal} that
on the spaces
$W_6:=\operatorname{SU}(3)/T_{\max}$,
$W_{12}:=\operatorname{Sp}(3)/\operatorname{Sp}(1)\times \operatorname{Sp}(1)\times \operatorname{Sp}(1)$ and
$W_{24}:=F_4/\operatorname{Spin}(8)$
the Ricci flow deforms \textit{certain}\/ positively curved metrics into metrics with mixed sectional curvature.
Metrics with positive sectional curvature may lose
the positivity of the Ricci curvature not only on $W_{12}$.
The results of~\cite[Theorem 3.1]{Bo} were covered
in~\cite[Theorem 9]{ChWal} by asserting such a property for the Wallach spaces $W_{12}$ and $W_{24}$: ``There exist homogeneous metrics of strictly positive sectional curvature on the 12 and 24 dimensional examples that deform under the Ricci flow to metrics with some negative Ricci curvature.''
However, the results of~\cite[Theorem 9]{ChWal} cannot be extended
to the  space~$W_{6}$. Some of them  may still lose positive sectional curvature on~$W_{6}$, however, all retain positive Ricci curvature as proved in~\cite[Theorem 8]{ChWal}: ``If~$g_0$ is a homogeneous Riemannian structure on a 6 dimensional example with strictly positive sectional curvature, then under the Ricci flow it retains strictly positive Ricci curvature.''

In 2016--2017, the mentioned results of~\cite{Bo, ChWal} were generalized
in~\cite{AN, Ab7} to a special class of GWS with coincided parameters $a_1=a_2=a_3=a\in (0,1/2)$ in which  $W_6, W_{12}$ and $W_{24}$ are contained as the members corresponding to the values $a=1/6$, $a=1/8$ and $a=1/9$.
Namely, for metrics~\eqref{metric} of the kind $x_1\ne x_2\ne x_3\ne x_1$ (called \textit{generic metrics}) we proved in~\cite[Theorem 1]{AN} that they \textit{all}\/  lose positivity of sectional curvature on the spaces $W_6, W_{12}$ and $W_{24}$; in~\cite[Theorem~4]{AN} a wider set of  metrics  was found on $W_6$ (the case $a=1/6$) preserving positive the Ricci curvature, even  losing the positivity of the sectional curvature; it was proved in~\cite[Theorem 3]{AN} and~\cite[Theorem 3]{Ab7}  that
\textit{all}\/  metrics~\eqref{metric} with positive Ricci curvature
are evolved into metrics with mixed Ricci curvature under the NRF for $a\in (0, 1/6)$, whereas the NRF evolves \textit{all}\/  metrics~\eqref{metric}  into metrics with
positive Ricci curvature in a finite time, if $a\in (1/6, 1/2)$. Note that some generalization of~\cite[Theorem~1]{AN} was obtained in~\cite{Ab_RM}.
Note also that metrics with $x_i=x_j$, $i\ne j$ (so called \textit{exceptional metrics}) admit a simple dynamics along invariant curves of the flow, and, therefore, are of no interest. Hence, in the sequel by ``a metric'' we mean ``a generic metric''.

Recently, in 2024,
some extended results on positive intermediate Ricci curvatures for an infinite series of homogeneous spaces
$$
\operatorname{Sp}(n+1)/\operatorname{Sp}(n-1)\times \operatorname{Sp}(1)\times \operatorname{Sp}(1),
$$
was obtained by Gonz\'alez-\'Alvaro and  Zarei in~\cite{Gonzales},
where the space $W_{12}$ is included as the first member corresponding to $n=2$.
By $d$-th intermediate Ricci curvature $\operatorname{Ric}_d$ they mean
$\operatorname{Ric}_d(v):=\sum_{i=1}^d\mathrm{K}(v,v_i)$,  $1\le d\le n-1$,
in addition, the  sectional curvature and the Ricci curvature
correspond to the values $d=1$ and $d=n-1$ respectively,
where $n$ is the dimension of the space.
A similar result was obtained by Cavenaghi et al. in~\cite{Caven} for a family of the homogeneous spaces
$$
\operatorname{SU}(m+2p)/\operatorname{S}(\operatorname{U}(m)\times \operatorname{U}(p)
\times \operatorname{U}(p)),
$$
where the space $W_{6}$ is contained at $m=p=1$.
Results of~\cite{Gonzales} and~\cite{Caven} related to the Ricci curvature,  were generalized in~\cite{Ab24} to  the spaces
\begin{eqnarray*}
&& \operatorname{SO}(k+l+m)/\operatorname{SO}(k)\times \operatorname{SO}(l)
\times \operatorname{SO}(m),\\
&& \operatorname{SU}(k+l+m)/\operatorname{S}(\operatorname{U}(k)\times \operatorname{U}(l) \times \operatorname{U}(m)), \\
&& \operatorname{Sp}(k+l+m)/\operatorname{Sp}(k)\times \operatorname{Sp}(l)
\times \operatorname{Sp}(m)
\end{eqnarray*}
with arbitrary  $k,l,m\in \mathbb{N}$.
In~\cite{Ab24} the other homogeneous spaces were also studied,  given in Table~1, according to the classification of Nikonorov~\cite{Nikonorov4}.
GWS~\textbf{n} in Table~1 means a generalized Wallach space~$G/H$ with corresponding Lie algebras $(\mathfrak{g}, \mathfrak{h})$ positioned in line~{\bf n}.  The parameter $\theta$ means the difference $(a_1+a_2+a_3)-1/2$.
Due to symmetry,  assume  $k\ge l\ge m\ge 1$ everywhere in the sequel.

\renewcommand{\arraystretch}{1.6}
\begin{center}
\hfill
\footnotesize{T\;a\;b\;l\;e\;\; 1}
\vskip 6pt
\textbf{The list of generalized Wallach spaces $G/H$ with  $G$ simple
according to~\cite{Nikonorov4}}
\vskip 6pt
\begin{tabular}{|c|c|c|c|c|c|c|}
\hline
GWS&$\mathfrak{g}$& $\mathfrak{h}$ & $a_1$ & $a_2$ & $a_3$ & $\theta$ \\\hline
${\bf 1}$&$\mathfrak{so}(k+l+m)$& $\mathfrak{so}(k)\oplus\mathfrak{so}(l)\oplus\mathfrak{so}(m)$          &$\frac{k}{2(k+l+m-2)}$  & $\frac{l}{2(k+l+m-2)}$&  $\frac{m}{2(k+l+m-2)}$&$\frac{1}{k+l+m-2}$\\
${\bf 2}$&$\mathfrak{su}(k+l+m)$& $\mathfrak{su}(k)\oplus\mathfrak{su}(l)\oplus\mathfrak{su}(m)$          & $\frac{k}{2(k+l+m)}$  & $\frac{l}{2(k+l+m)}$&  $\frac{m}{2(k+l+m)}$&$0$ \\
{\bf 3}&$\mathfrak{sp}(k+l+m)$& $\mathfrak{sp}(k)\oplus\mathfrak{sp}(l)\oplus\mathfrak{sp}(m)$          &$\frac{k}{2(k+l+m+1)}$&$\frac{l}{2(k+l+m+1)}$&  $\frac{m}{2(k+l+m+1)}$&$\frac{-1}{2(k+l+m+1)}$\\
{\bf 4}&$\mathfrak{su}(2l),~ l\ge 2$& $\mathfrak{u}(l)$
& $\frac{l+1}{4l}$  & $\frac{l-1}{4l}$&  $\frac{1}{4}$ &$1/4$  \\
{\bf 5}&$\mathfrak{so}(2l),~ l\ge 4$& $\mathfrak{u}(1)\oplus \mathfrak{u}(l-1)$
& $\frac{l-2}{4(l-1)}$  & $\frac{l-2}{4(l-1)}$&  $\frac{1}{2(l-1)}$&$0$   \\
{\bf 6}&$\mathfrak{e}_6$& $\mathfrak{su}(4)\oplus 2\mathfrak{sp}(1)\oplus \mathbb{R}$
& $1/4$  & $1/4$&  $1/6$  &$1/6$ \\
{\bf 7}&$\mathfrak{e}_6$& $\mathfrak{so}(8)\oplus  \mathbb{R}^2$
& $1/6$  & $1/6$&  $1/6$  &$0$  \\
{\bf 8}&$\mathfrak{e}_6$& $\mathfrak{sp}(3)\oplus  \mathfrak{sp}(1)$
& $1/4$  & $1/8$&  $7/24$  &$1/6$ \\
{\bf 9}&$\mathfrak{e}_7$& $\mathfrak{so}(8)\oplus  3\mathfrak{sp}(1)$
& $2/9$  & $2/9$&  $2/9$ &$1/6$  \\
{\bf 10}&$\mathfrak{e}_7$& $\mathfrak{su}(6)\oplus \mathfrak{sp}(1)\oplus \mathbb{R}$
& $2/9$  & $1/6$&  $5/18$  &$1/6$ \\
{\bf 11}&$\mathfrak{e}_7$& $\mathfrak{s0}(8)$
& $5/18$  & $5/18$&  $5/18$  &$1/3$ \\
{\bf 12}&$\mathfrak{e}_8$& $\mathfrak{so}(12)\oplus  2\mathfrak{sp}(1)$
& $1/5$  & $1/5$&  $4/15$  &$1/6$ \\
{\bf 13}&$\mathfrak{e}_8$& $\mathfrak{so}(8)\oplus  \mathfrak{so}(8)$
& $4/15$  & $4/15$&  $4/15$  &$3/10$ \\
{\bf 14}&$\mathfrak{f}_4$& $\mathfrak{so}(5)\oplus  2\mathfrak{sp}(1)$
& $5/18$  & $5/18$&  $1/9$ &$1/6$  \\
{\bf 15}&$\mathfrak{f}_4$& $\mathfrak{so}(8)$
& $1/9$  & $1/9$&  $1/9$  &$-1/6$ \\\hline
\end{tabular}
\end{center}

\medskip

In~\cite{Ab24} the results of~\cite[Theorems~2--4]{AN} and~\cite[Theorem~3]{Ab7} were also extended to the general case  $a_1, a_2, a_3\in (0,1/2)$.
Briefly, the questions were  clarified in~\cite{Ab7, AN, Ab24}: on which spaces can the positivity of the Ricci curvature be preserved or lost
and does such a property hold for all metrics or only for a part of them. Another interesting question arises about the number of such spaces.

\smallskip

Let $\textbf{x}=(x_1,x_2,x_3)$,
$\textbf{a}=(a_1,a_2,a_3)$ and  $X^3=X\times X\times X$ for any set $X$.
Introduce  the cones~$\Lambda_i$ defined by ${\bf r}_i=0$ and the set
$$
\mathscr{R}_{+}=\big\{\textbf{x}\in (0,+\infty)^3 ~~ \big|~~ {\bf r}_1>0, ~  {\bf r}_2>0, ~  {\bf r}_3>0\big\},
$$
where ${\bf r}_1, {\bf r}_2$ and ${\bf r}_3$ are given in~\eqref{Ricci_princ}.
Then $\operatorname{Ric}_{{\bf g}(t)}>0$ (or briefly $\operatorname{Ric}>0$) for~\eqref{ricciflow}
is equivalent to $\textbf{x}(t)\in \mathscr{R}_{+}$ for~\eqref{three_equat} according to~\eqref{tensor}.
Respectively, in the sequel any sentence containing the phrase ``the positivity of the Ricci curvature is preserved (under the NRF)'' will mean $\operatorname{Ric}_{{\bf g}(t)}>0$   (equivalently $\textbf{x}(t)\in \mathscr{R}_{+}$) for all $t>0$ and \textit{any} original metric ${\bf g}(0)$ with $\operatorname{Ric}_{{\bf g}(0)}>0$  (equivalently  $\textbf{x}(0)\in \mathscr{R}_{+}$); in such cases we will also use the even shorter phrase ``\textit{all} metrics with $\operatorname{Ric}>0$ retain $\operatorname{Ric}>0$''.
Analogously, ``the positivity of the Ricci curvature can not be preserved'' will mean existing of \textit{some}  metric with $\operatorname{Ric}_{{\bf g}(0)}>0$ ($\textbf{x}(0)\in \mathscr{R}_{+}$) and existing of a finite time $t_0>0$ such that
$\operatorname{Ric}_{{\bf g}(t)}>0$ fails ($\textbf{x}(t)\notin \mathscr{R}_{+}$) for all $t\ge t_0$; a shorter phrase is ``\textit{some} metrics with $\operatorname{Ric}>0$ lose $\operatorname{Ric}>0$''.
Let's use ``the positivity of the Ricci curvature is lost'' in the strictest sense meaning that for \textit{every}  metric with $\operatorname{Ric}_{{\bf g}(0)}>0$ ($\textbf{x}(0)\in \mathscr{R}_{+}$) there exists a finite time $t_0>0$ such that
$\operatorname{Ric}_{{\bf g}(t)}>0$ fails ($\textbf{x}(t)\notin \mathscr{R}_{+}$) for all $t\ge t_0$; a shorter phrase is ``\textit{all} metrics with $\operatorname{Ric}>0$ lose $\operatorname{Ric}>0$''.
We need also the invariant surface $\Sigma$ of the system~\eqref{three_equat} defined by the equation $x_1^{1/a_1}x_2^{1/a_2}x_3^{1/a_3}=1$.

\smallskip

 The main results of this article are described in
 Theorems~\ref{thm0_KarSU_25}\,--\,\ref{thm2_KarSU_25}.

\begin{theorem}\label{thm0_KarSU_25}
There are infinitely many generalized Wallach spaces
such that every metric~\eqref{metric} on every of them loses
the positivity of the Ricci curvature when evolved by the normalized Ricci flow~\eqref{ricciflow}.
\end{theorem}

\begin{theorem}\label{thm1_KarSU_25}
There are infinitely many generalized Wallach spaces
such that every metric~\eqref{metric} on every of them preserves
the positivity of the Ricci curvature when evolved by the normalized Ricci flow~\eqref{ricciflow}.
\end{theorem}

\begin{theorem}\label{thm2_KarSU_25}
Assume that the normalized Ricci flow~\eqref{ricciflow} evolves metrics~\eqref{metric}
on generalized Wallach spaces with  coincided parame\-ters $a_1=a_2=a_3:=a \in (0, 1/2)$.
Then
\begin{enumerate}
\item [1)]
All metrics lose $\operatorname{Ric}>0$,~ if  $a\in (0,1/6)$;
\item [2)]
Some metrics preserve $\operatorname{Ric}>0$,~ if $a\in [1/6, a^{\ast})$;
\item [3)]
All metrics preserve $\operatorname{Ric}>0$,~ if  $a\in [a^{\ast}, 1/2)$,
\end{enumerate}
where  $a^{\ast}$ is some number from the interval $(1/6, 1/4)$.
\end{theorem}

The exact value of  $a^{\ast}$ is
$$
a^{\ast}=(\mu-1)^2(12\mu)^{-1}
$$
with $\mu=\sqrt[3]{28+3\sqrt{87}}$.
Approximately, $a^{\ast}\approx 0.1739051924$ (see Remark~\ref{exact_a}).

\section{The proofs of the main results}

\begin{proof}[Proof of Theorem~\ref{thm0_KarSU_25}]
According to Theorem~6 in~\cite{Ab24}
the positivity of the Ricci curvature of metrics~\eqref{metric} can not be preserved
on a GWS with $a_1+a_2+a_3< 1/2$, when evolving by the NRF~\eqref{ricciflow}.
An infinite number of homogeneous spaces
$\operatorname{Sp}(k+l+m)/\operatorname{Sp}(k)\times \operatorname{Sp}(l)
\times \operatorname{Sp}(m)$ with $k,l,m\in \mathbb{N}$  satisfies that theorem, since
$$
a_1+a_2+a_3-\frac{1}{2}=-\frac{1}{2(k+l+m+1)}<0
$$
due to
$a_1=\frac{k}{2(k+l+m+1)}$,
$a_2=\frac{l}{2(k+l+m+1)}$ and
$a_3=\frac{m}{2(k+l+m+1)}$
(see line~\textbf{3} in Table~1 in~\cite{Nikonorov4}).
For $k=l=m$ we have an infinite family
$\operatorname{Sp}(k)/\operatorname{Sp}(k)\times \operatorname{Sp}(k)
\times \operatorname{Sp}(k)$ with $a_1=a_2=a_3=a= \frac{k}{6k+2}<\frac{1}{6}$.
According to~\cite[Theorem~3]{AN} every metric loses the positivity of the Ricci
curvature on every such space.
\end{proof}

\bigskip

\begin{proof}[Proof of Theorem~\ref{thm1_KarSU_25}]
Let us introduce the following sets of parameters $(a_1,a_2,a_3)\in (0,1/2)^3$:
$$
\mathscr{A}=\big\{a_1+a_2+a_3>1/2\,
\mbox{~and the trajectories of~\eqref{three_equat}~do not leave}~\mathscr{R}_{+} \big\},
$$
$$
\mathscr{B}=\big\{(a_1,a_2,a_3)\in \mathscr{A} \mbox{~corresponds to a certain GWS}\big\}.
$$

Let us first prove that the set \textit{$\mathscr{A}$ is infinite}.
As follows from Theorem~~\cite[Theorem~7]{Ab24}
$(a_1,a_2,a_3)\in\mathscr{A}$ is ensured if
$\theta \ge \max\left\{\theta_1, \theta_2, \theta_3\right\}$,
where $\theta =a_1+a_2+a_3-1/2$.
We first establish the estimate
\begin{equation}\label{ineqTeta}
0<\theta_i(a_i)<1/6
\end{equation}
for each  function
$$
\theta_i=\theta_i(a_i)=a_i-\dfrac{1}{2}+\dfrac{1}{2}\sqrt{\dfrac{1-2a_i}{1+2a_i}}
$$
defined for $a_i\in (0,1/2)$, where
$i=1,2,3$.
Indeed  $\theta_i(0)=\theta_i(1/2)=0$ and for all $a_i\in (0,1/2)$ and every $i=1,2,3$, see Figure~\ref{improv71}.
Indeed~\eqref{ineqTeta} is equivalent to\,
$$
3\sqrt{(1-2a)(1+2a)^{-1}}<2(2-3a),
$$
where $a:=a_i\in (0,1/2)$.
Since $1-2a>0$ and $2-3a>0$ for  $a\in (0,1/2)$,
then  raising the latter inequality to square
and simplifying, we obtain a new inequality
$$
-72a^3+60a^2-2a-7<0,
$$
equivalent to~\eqref{ineqTeta}.
Introduce the cubic  polynomial
$p(a)=-72a^3+60a^2-2a-7$.
Negativity of its determinant  implies
that $p(a)$ admits one real root (and two complex roots).
Since $p(-\infty)>0$
\big(where $p(-\infty):=\lim_{a\to -\infty}p(a)=+\infty$\big)
and $p(0)<0$, that  real root must be negative.
In what follows that  $p(a)<0$ for all $a\in (0,+\infty)$, consequently,
for all $a\in (0,1/2)$. This justifies the inequality~\eqref{ineqTeta}.

The estimate~\eqref{ineqTeta} provides an opportunity to make~$\theta >\theta_i$ for each~$i=1,2,3$ by choosing infinitely many values of the parameters~$a_i$ from the interval $(2/9, 1/2)$. Indeed
$$
\theta=a_1+a_2+a_3-\frac{1}{2}> 3\cdot \frac{2}{9}-\frac{1}{2}= \frac{2}{3}-\frac{1}{2}=\frac{1}{6}> \theta_i
$$
and, hence, the conditions $a_1, a_2, a_3\in (2/9, 1/2)$ are sufficient to ensure that no trajectory of the differential system~\eqref{three_equat}  leaves the domain~$\mathscr{R}_{+}$, starting in~$\mathscr{R}_{+}$.
Thus, $\mathscr{A}$ can be chosen equal to the infinite (uncountable) set $(2/9, 1/2)^3$.

\smallskip

\textit{$\mathscr{B}$ is an infinite set}.
It is well known that every GWS is described by some triple
$(a_1, a_2, a_3)\in (0,1/2)^3$, but not every  triple
can correspond to a certain GWS (see~\cite{Nikonorov4}).
Therefore, we need to prove the existence of  infinitely many triples $(a_1, a_2, a_3)\in \mathscr{A}$ that correspond to actual GWSs.
A suitable example for this aim is an infinite family of  homogeneous spaces $\operatorname{SU}(2l)/\operatorname{U}(l)$, $l\ge 2$
(see line~\textbf{4} in~\cite[Table~1]{Nikonorov4})
with
$$
a_1=\frac{l+1}{4l}, \quad a_2=\frac{l-1}{4l}, \quad a_3=\frac{1}{4}
$$
and
$$
\theta=a_1+a_2+a_3-1/2=1/4>1/6>\theta_i,  \qquad i=1,2,3.
$$
Therefore, for $\mathscr{B}$ we can choose the countable set
$$
\mathscr{B}:=\big\{\big(l+1)/4, (l-1)/4, 1/4\big)~ |~ l\ge 2\big\} \subset \mathscr{A}.
 $$
 This means that there is an
 infinite number of GWSs on which any metric~\eqref{metric} with
 $\operatorname{Ric}>0$ maintains  $\operatorname{Ric}>0$
 when evolved by the NRF~\eqref{ricciflow}.
Theorem~\ref{thm1_KarSU_25} is proved.
\end{proof}

\begin{remark}
Unfortunately, the homogeneous spaces  $\operatorname{SO}(k+l+m)/\operatorname{SO}(k)\times \operatorname{SO}(l)
\times \operatorname{SO}(m)$  (see row~\textbf{1} in~Table~2),
can  not satisfy~Theorem~\ref{thm1_KarSU_25}.
According to~\cite[Theorem~8]{Ab24}
there is only a finite number of such spaces
on which the NRF~\eqref{ricciflow} preserves $\operatorname{Ric}>0$
for every initial metric~\eqref{metric} with  $\operatorname{Ric}>0$.
The extra condition $l\ge 2$ was supposed
to except $l=m=1$  which yields $(a_1,a_2,a_3)\notin (0,1/2)^3$
due to $a_1=1/2$. For this reason the case $l=m=1$  was studied in~\cite[Theorem~4]{Ab242} separately, where it was proved  that the NRF preserves the property $\operatorname{Ric}>0$ on  every Stiefel manifold
$\operatorname{SO}(n)/\operatorname{SO}(n-2)$ with $n\ge 3$.
\end{remark}

\begin{remark} The estimate $\theta_i(a)<1/6$ can be improved by the sharpest upper bound of the function~$\theta_i(a)$. Indeed
$$
\frac{d \theta_i}{da}=1-\frac{\sqrt{v}}{(1+2a)^2},
$$
where $v=(1+2a)(1-2a)^{-1}>0$ for all $a\in (0,1/2)$.
Then
$$
\dfrac{d \theta_i}{da}\ge 0   \Leftrightarrow
v\le (1+2a)^4
\Leftrightarrow  4a^3+4a^2-1\le 0.
$$
The only real root $a_0\in (0.4, 0.5)$
of the cubic polynomial $4a^3+4a^2-1$ equals approximately
$0.4196$. Then
$$
\theta_i(a)\le\theta_i(a_0)=\max_{a\in [0,1/2]} \theta_i(a)\approx 0.0674,
$$
see Figure~\ref{improv71}.
\end{remark}

\begin{figure}[h]
\centering
\includegraphics[width=0.5\linewidth]{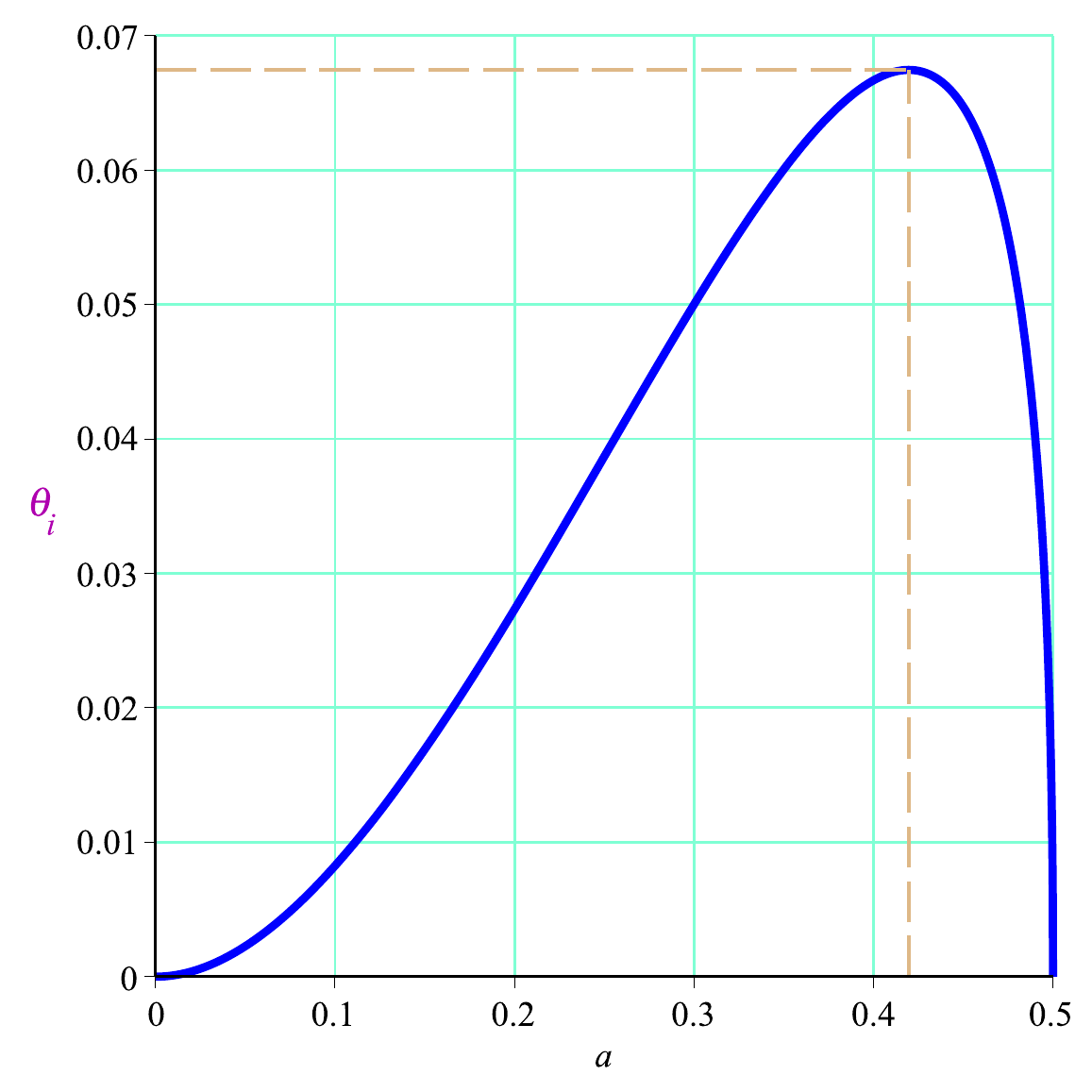}
\caption{The function $\theta_i(a)$}
\label{improv71}
\end{figure}

\begin{proof}[Proof of Theorem~\ref{thm2_KarSU_25}]
Since $1-2a_i>0$ for $a_i\in (0,1/2)$
the condition  $\theta\ge \theta_i$ in~\cite[Theorem~7]{Ab24} is equivalent to
$$
4(a_j+a_k)^2\ge (1-2a_i)(1+2a_i)^{-1},
$$
where $\theta=a_i+a_j+a_k-1/2$ since
$(i,j,k)$ is any permutation of $(1,2,3)$.
For  $a_1=a_2=a_3=a$ this equivalent to
$$
32a^3+16a^2+2a-1\ge 0.
$$
The cubic polynomial $h(a):=32a^3+16a^2+2a-1$ has a real root $a=a^{\ast}$ between
$a=1/6$ and $a=1/4$, because of $h(1/6)=-2/27<0$ and  $h(1/4)=1>0$.
That root is unique (excepting two complex roots)
by  negativity  of the discriminant.
Hence
$$
\theta\ge \theta_1=\theta_2=\theta_3
$$
for all
$a\in [a^{\ast}, 1/2)$.
This means that
every initial metric with $\operatorname{Ric}>0$
retain $\operatorname{Ric}>0$, if $a\in [a^{\ast}, 1/2)$.
As for the case $a\in (0, a^{\ast})$,
then
$$
\theta< \theta_1=\theta_2=\theta_3
$$
and~\cite[Theorem~7]{Ab24} stated
that at least some metrics can maintain ${\operatorname{Ric}>0}$.
We can give a more detailed analysis to the interval~$(0, a^{\ast})$ now.
Indeed according to~\cite[Theorem~3]{AN} any  metric with $\operatorname{Ric}>0$ loses $\operatorname{Ric}>0$ on any GWS with $a\in (0, 1/6)$.
Then due to the inequality $a^{\ast}>1/6$ shown above,
the less strict conclusion of~\cite[Theorem~7]{Ab24}
continues to hold for the values
$a\in (1/6, a^{\ast})$. Theorem~\ref{thm2_KarSU_25} is proved.
\end{proof}

\begin{remark}\label{exact_a}
The exact value of~$a^{\ast}$ can be found by Cardano's formula.
Indeed replacing $a=b-1/6$ we reduce $32a^3+16a^2+2a-1=0$ to a
cubic equation $b^3+pb+q=0$ in the canonical form, where  $p=-1/48$ and $q=-7/216$,
that admits the unique real root
$\widetilde{b}=\alpha+\beta$, where
$$
\alpha=\sqrt[3]{-q/2+\sqrt{Q}}, \qquad
\beta=\sqrt[3]{-q/2-\sqrt{Q}}
$$
and  $Q=(p/3)^3+(q/2)^2$.
Since $\alpha\beta=-p/3$ it follows then
$$
a^{\ast}=\widetilde{b}-\frac{1}{6}=\alpha+\frac{1}{144\alpha}-\frac{1}{6}=
\frac{(\mu-1)^2}{12\mu}
$$
with $\mu=12\alpha$.
\end{remark}

\section{Examples}

\begin{figure}[h]
\centering
\includegraphics[width=0.5\linewidth]{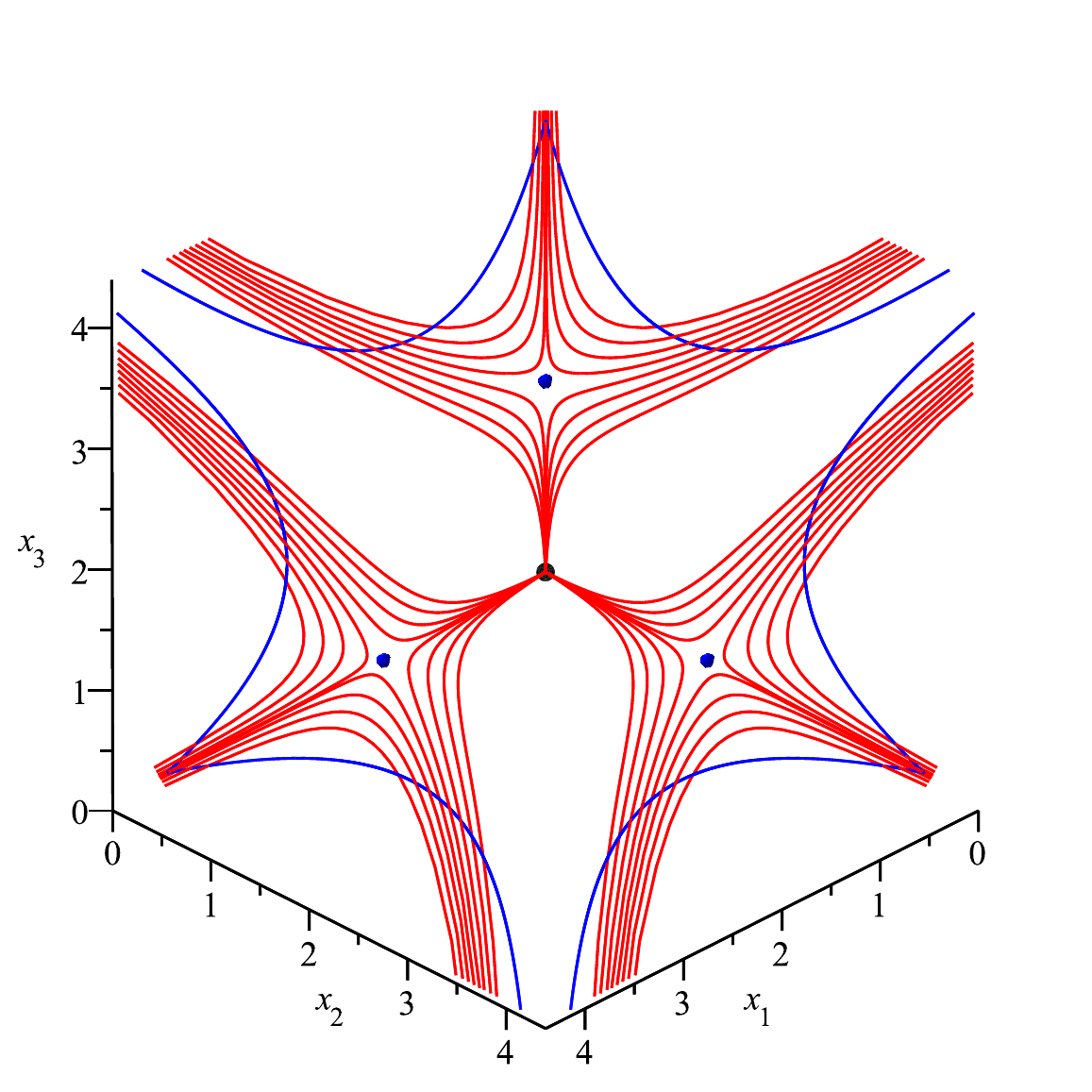}
\caption{$F_4/\operatorname{Spin}(8)$: trajectories leave~$\Sigma\cap \mathscr{R}_{+}$}
\label{improv72}
\end{figure}

\begin{figure}[h]
\centering
\includegraphics[width=0.5\linewidth]{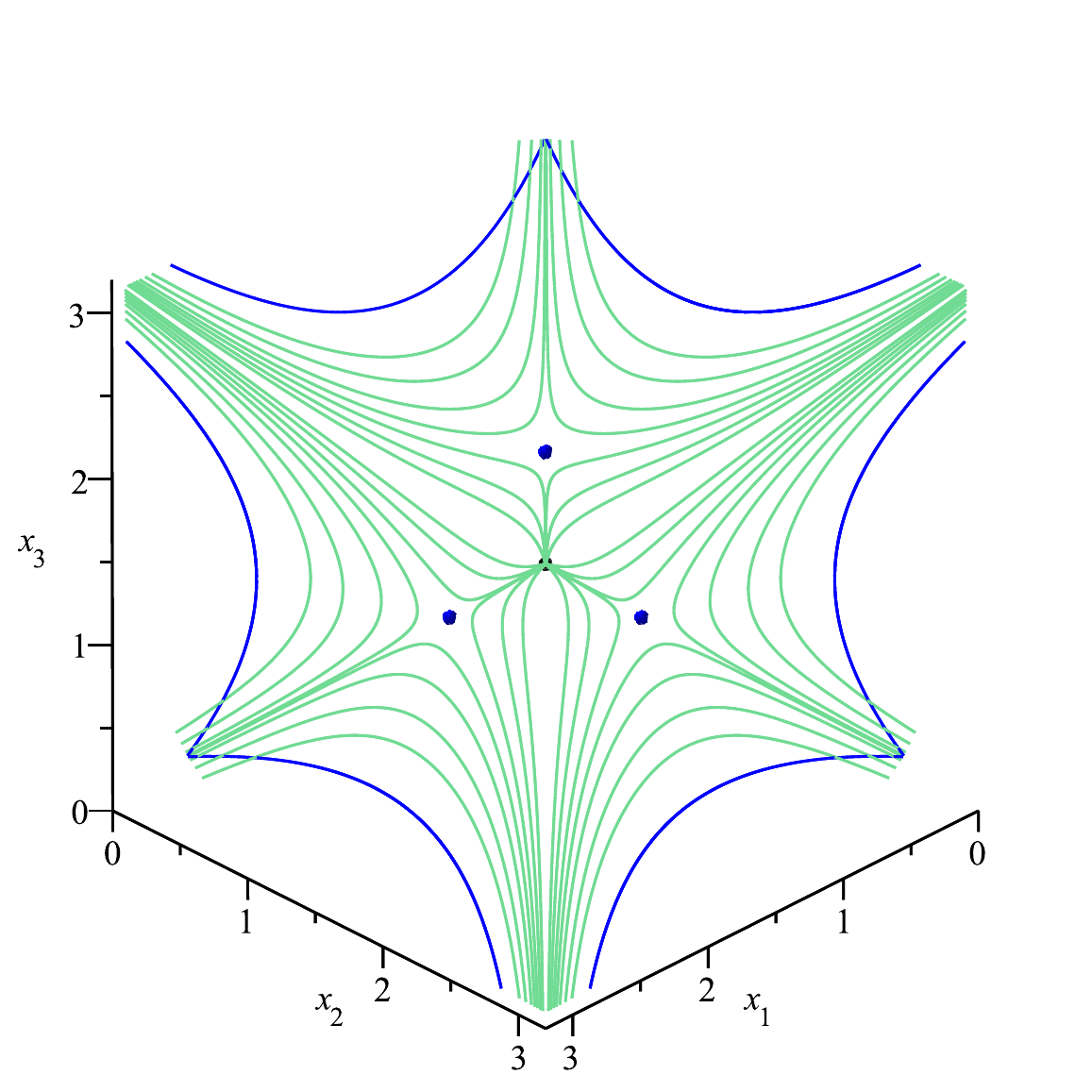}
\caption{The NRF~\eqref{ricciflow} on GWS~\textbf{1} with $k=l=m=16$ ($a=4/23$):
trajectories stay in~$\Sigma\cap \mathscr{R}_{+}$}
\label{improv73}
\end{figure}

\begin{example}
\textit{The case $a\in (0,1/6)$}.
Any metric~\eqref{metric}  with $\operatorname{Ric}>0$ loses  $\operatorname{Ric}>0$ on every space\,
$\operatorname{Sp}(3k)/\operatorname{Sp}(k)\times \operatorname{Sp}(k)
\times \operatorname{Sp}(k)$ (GWS~\textbf{3} in Table~1),
when evolved by the NRF~\eqref{ricciflow}.
\end{example}

Note that
$a_1=a_2=a_3=\dfrac{k}{6k+2}<1/6$,
satisfying Theorem~\ref{thm2_KarSU_25} for all\, $k\in \mathbb{N}$.

\begin{example}
The  Wallach space ${W_{24}=F_4/\operatorname{Spin}(8)}$ (GWS~\textbf{15})
is another example of a homogeneous space on which
every metric with $\operatorname{Ric}>0$ loses  $\operatorname{Ric}>0$ under the NRF~\eqref{ricciflow}.
\end{example}

Indeed
$a_1=a_2=a_3=a=1/9$ for~$W_{24}$ and, hence, $a<1/6$, satisfying Theorem~\ref{thm2_KarSU_25}.
Trajectories (in red) of the system~\eqref{three_equat} correspon\-ding to~$W_{24}$
are depicted in Figure~\ref{improv72} that leave the domain $\Sigma\cap \mathscr{R}_{+}$  and never return back there.

\begin{example}\label{ex3}
\textit{The case $a\in \big[a^{\ast}, 1/2\big)$}.
Every metric~\eqref{metric}  with $\operatorname{Ric}>0$ maintains  $\operatorname{Ric}>0$ on every  space\,
$\operatorname{SO}(3k)/\operatorname{SO}(k)\times \operatorname{SO}(k)
\times \operatorname{SO}(k)$
with $k\in \{2,\dots, 16\}$, when evolved by the NRF~\eqref{ricciflow}.
 \end{example}

Indeed  $1/6< a \le 1/4$\, for all\, $k\ge 2$ at $a=a_1=a_2=a_3=k/(6k-4)$.
Admitting solutions
$$
k\le \dfrac{4a^{\ast}}{6a^{\ast}-1}\approx 16.0166,
$$
the inequality $a\ge a^{\ast}$ can be satisfied only for $k\in \{2,\dots, 16\}$
with $a^{\ast}\approx 0.1739051924$.
The case~$k=16$ is illustrated in Figure~\ref{improv73}, where trajectories (in aquamarine) all remain in the domain~$\Sigma\cap \mathscr{R}_{+}$ due to
$$
a=\frac{k}{6k-4}=\frac{4}{23}>a^{\ast},
$$
where  $4/23\approx 0.1739130435$.
Note that  $k\le 16$  is consistent with~\cite[Theorem~7]{Ab24}.

\begin{example}\textit{The case $a\in \big[a^{\ast}, 1/2\big)$}.
$\operatorname{Ric}>0$ is preserved on the spaces
$E_7/\operatorname{SO}(8)\times \operatorname{SU}(2)\times \operatorname{SU}(2)
\times \operatorname{SU}(2)$,
$E_7/\operatorname{SO}(8)$
and
$E_8/\operatorname{SO}(8)\times \operatorname{SO}(8)$
located in lines \textbf{9, 11} and~\textbf{13} in Table~1.
 \end{example}

The following values of $a$ correspond to them respectively:
$$
a=2/9\approx 0.22>a^{\ast}, \quad
a=5/18\approx 0.27 >a^{\ast}, \quad a=4/15\approx 0.26>a^{\ast}.
$$

\begin{figure}[h]
\centering
\includegraphics[width=0.5\linewidth]{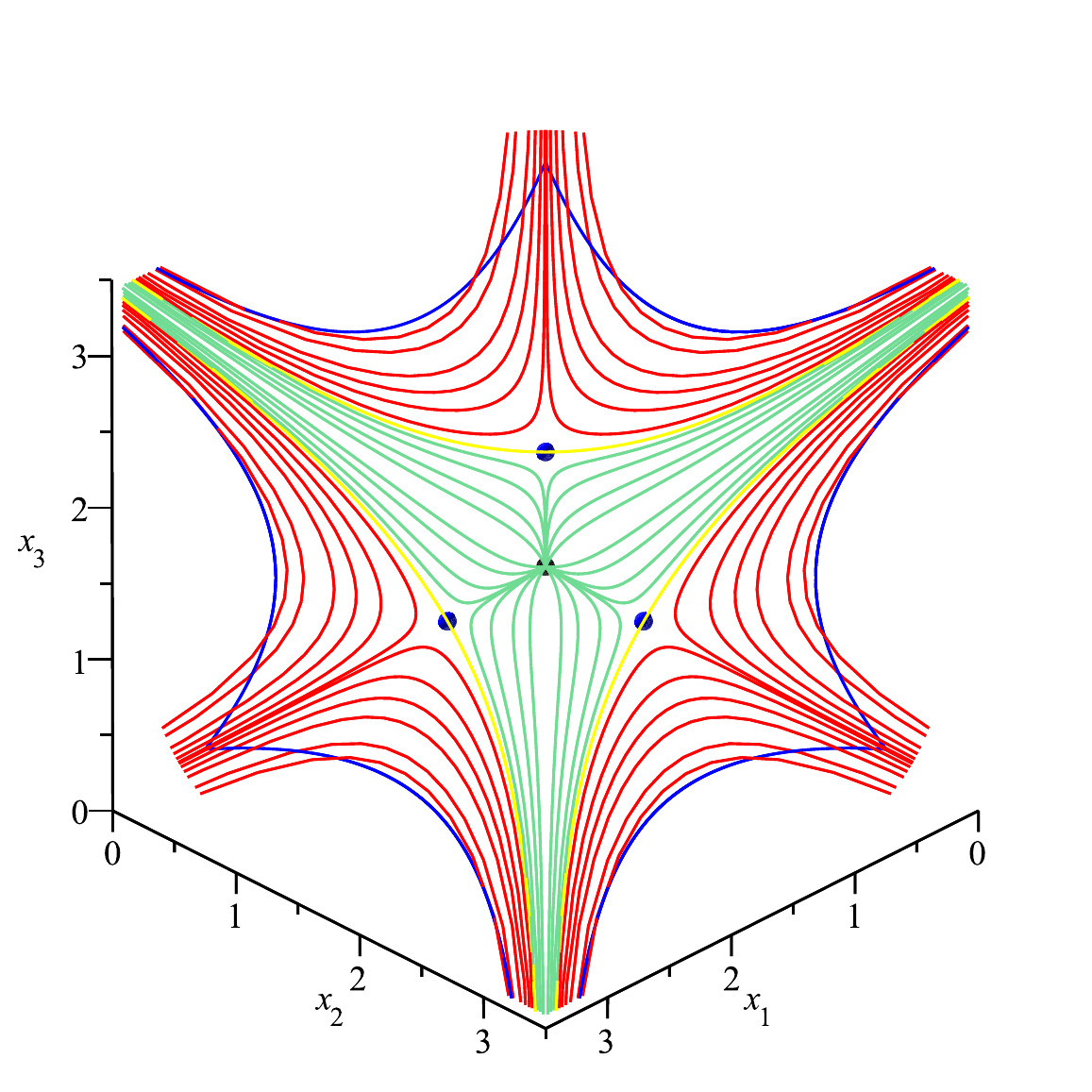}
\caption{The NRF~\eqref{ricciflow} on GWS with $a=1/6$:
there are trajectories both leaving and staying in~$\Sigma\cap \mathscr{R}_{+}$}
\label{improv74}
\end{figure}

\begin{figure}[h]
\centering
\includegraphics[width=0.5\linewidth]{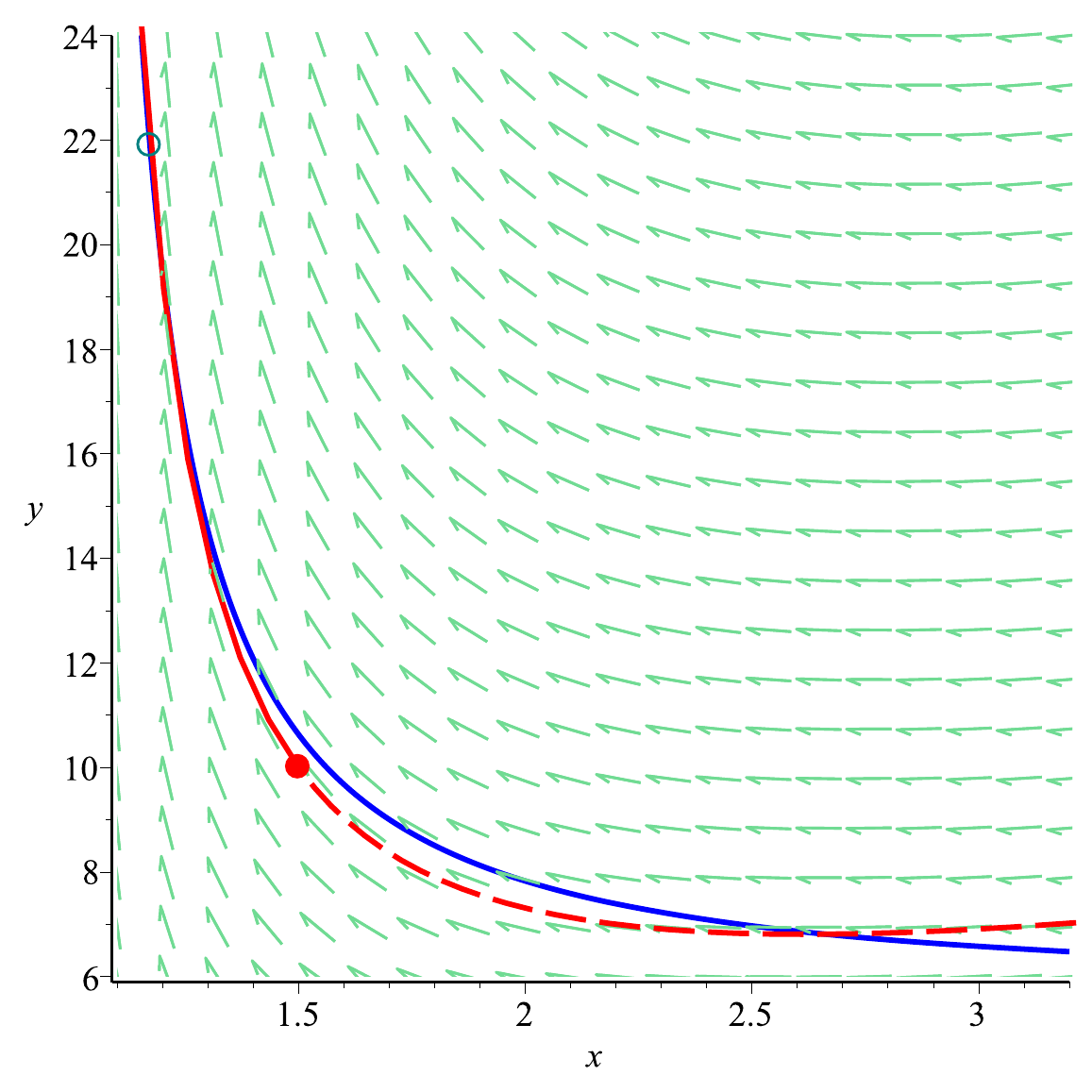}
\caption{The border curve $\gamma$ and the
approximated integral curve with the initial point $(x_0,y_0)=(1.5, 10)$ in the  case $a=1/6$:  part I}
\label{improv742}
\end{figure}

\begin{figure}[h]
\centering
\includegraphics[width=0.5\linewidth]{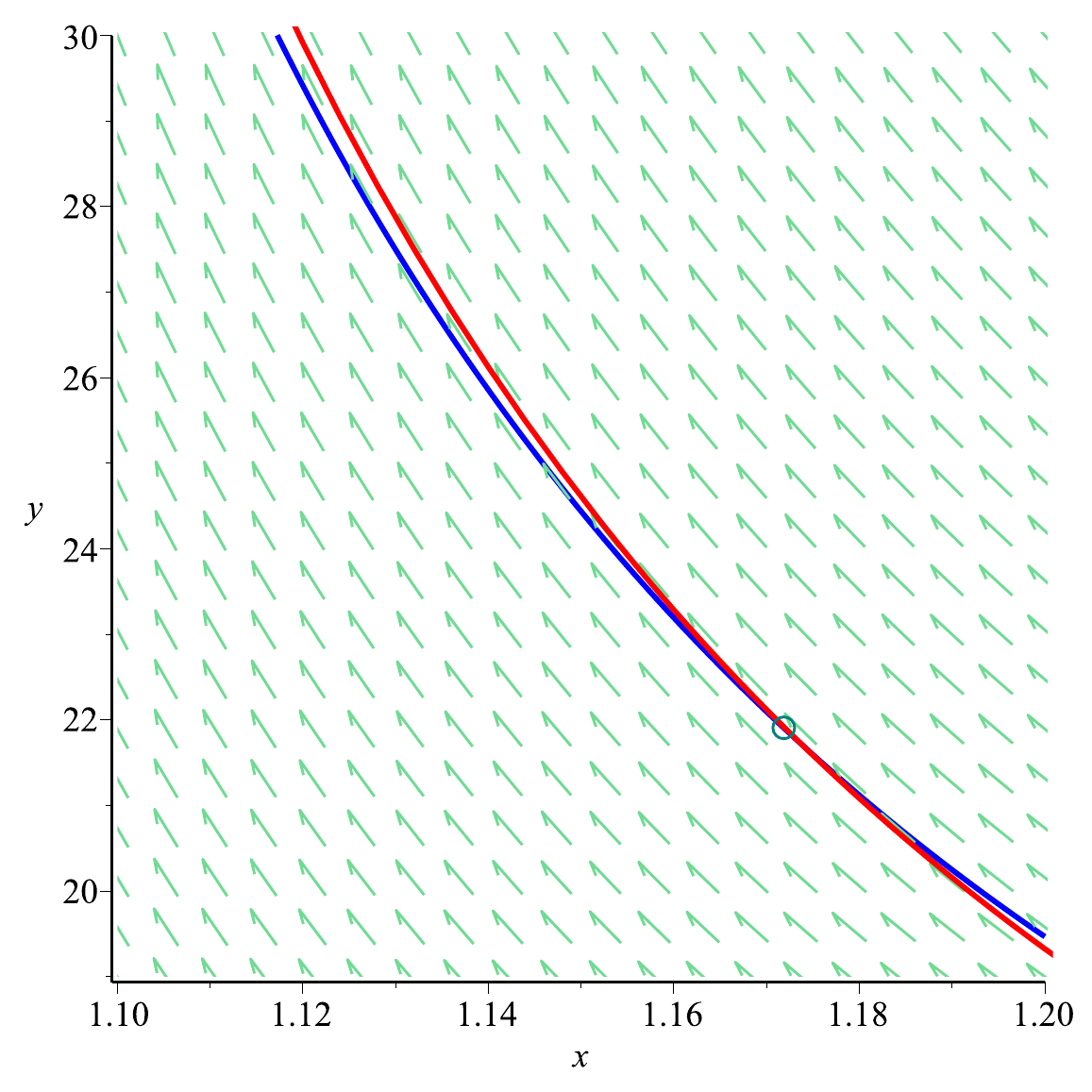}
\caption{The border curve $\gamma$ and the
approximated integral curve with the initial point $(x_0,y_0)=(1.5, 10)$ in the  case $a=1/6$:   part II}
\label{improv743}
\end{figure}

\begin{example}\textit{The case $a=1/6$}.
There  are   metrics both preserving $\operatorname{Ric}>0$ and losing $\operatorname{Ric}>0$ on every GWS with $a=1/6$. In Figure~\ref{improv74}  trajectories are illustrated
both leaving~$\Sigma\cap \mathscr{R}_{+}$ and retaining in~$\Sigma\cap \mathscr{R}_{+}$.
 \end{example}

Note that infinitely many GWS correspond to $a=1/6$, for instance,
so are $\operatorname{SU}(3k)/\operatorname{S}(\operatorname{U}(k)\times \operatorname{U}(k) \times \operatorname{U}(k))$.
According to~\cite[Theorem~4]{AN} those metrics \textit{retain $\operatorname{Ric}>0$}
 chosen from the domain bounded by the  curves
$$
\{x_k=x_i+x_j\} \cap \{x_ix_jx_k=1\},
$$
depicted in Figure~\ref{improv74} in yellow.
Metrics~\eqref{metric} with   $x_k=x_i+x_j$ are known as K\"ahler metrics on GWS.
Theorem~6 in~\cite{Ab24} implies also the existence of metrics~\eqref{metric}
losing $\operatorname{Ric}>0$.

To construct more visible illustrations of metrics  \textit{losing $\operatorname{Ric}>0$} it is  convenient to pass to new variables
$x=x_3/x_1$ and $y=x_3/x_1$ satisfying $y>x>1$ (see details in~\cite{AN, Ab24}). The border curve, for example,
$\zeta_1=\Sigma \cap\Lambda_1$ will be transformed to the curve
$$
y=\psi(x)= \frac{x+\sqrt{(1-4a^2)x^2+4a^2}}{2a(x^2-1)}\,x, \qquad  1<x<1/a.
$$
Let's denote it by  $\gamma$ (the curve in blue in Figures~\ref{improv742}
and~\ref{improv743}).

An arbitrary initial point  can be chosen from the  image $R$ of $\Sigma\cap \mathscr{R}_{+}$ under the homeomorphism $(x_1,x_2,x_3)\mapsto (x,y)$, so that the pre-image of that point belongs to~$\Sigma\cap \mathscr{R}_{+}$. For example,
$$
(x_0,y_0)=(1.5, 10)\in R
$$
satisfies such a condition, see Figure~\ref{improv742}.
Now the unique solution $y=\phi(x)$ of the Cauchy problem
$$
\frac{dy}{dx}=\frac{(y-1)(y-2axy-2ax)}{(x-1)(x-2axy-2ay)}, \qquad y(x_0)=y_0
$$
for $a=1/6$ (see~\cite{AN}),  can be approximately found in order to show that the trajectory
initiated in $(x_0, y_0)$ leaves the domain~$R$.
For this goal choose  the left border $b=1.1$ and $N=100$ mesh points
constructing a grid with the nodes $x_i=x_0+ih$,   $i=0,\dots, N$,
and the constant negative step size $h=(b-x_0)/N=-0.004$.
Using the fourth order Runge-Kutta method we observe
that the first positive difference
$$
[\phi]_i-\psi(x_i)>0
$$
happens at the mesh point $x_i=1.172$ with $i=82$, where $[\phi]_i$ is the approximative value of $\phi(x_i)$ evaluated by the finite difference scheme. The corresponding value of $\psi(x_i)$ is $21.8927$.
In Figures~\ref{improv742}
and~\ref{improv743} the Runge-Kutta approximation (in red) to the unique solution  of the Cauchy problem with $(x_0,y_0)=(1.5, 10)$ (the  solid circle in red)
is depicted that attains the border $\gamma$ (in blue) at the point (the circle in teal) $$
(x,y)\approx (1.172, 21.8927)
$$
and then leaves~$R$.

\begin{figure}[h]
\centering
\includegraphics[width=0.5\linewidth]{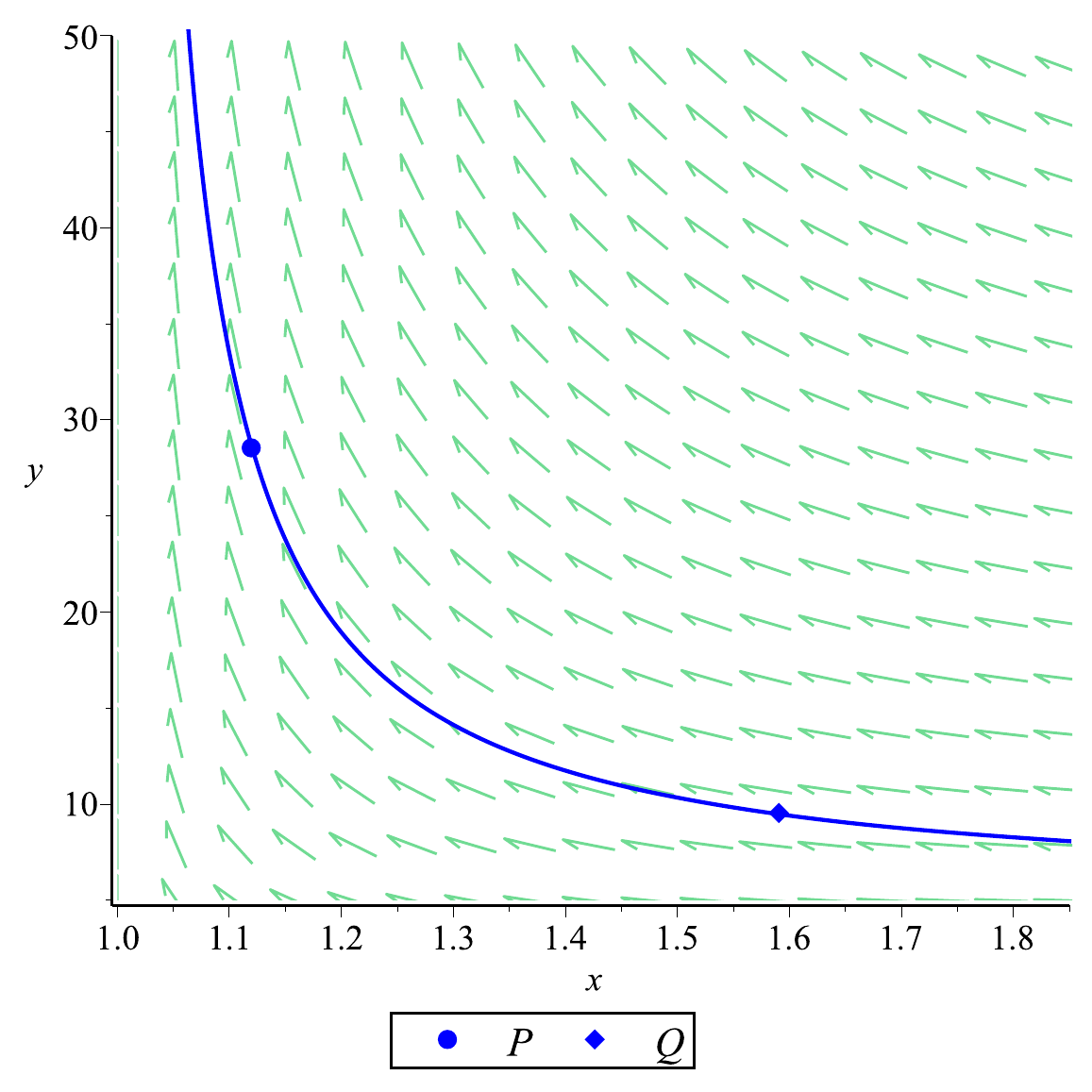}
\caption{Part I: The NRF~\eqref{ricciflow} on GWS~\textbf{1} with $k=l=m=25$
($a=25/146$)}
\label{improv75}
\end{figure}

\begin{figure}[h]
\centering
\includegraphics[width=0.5\linewidth]{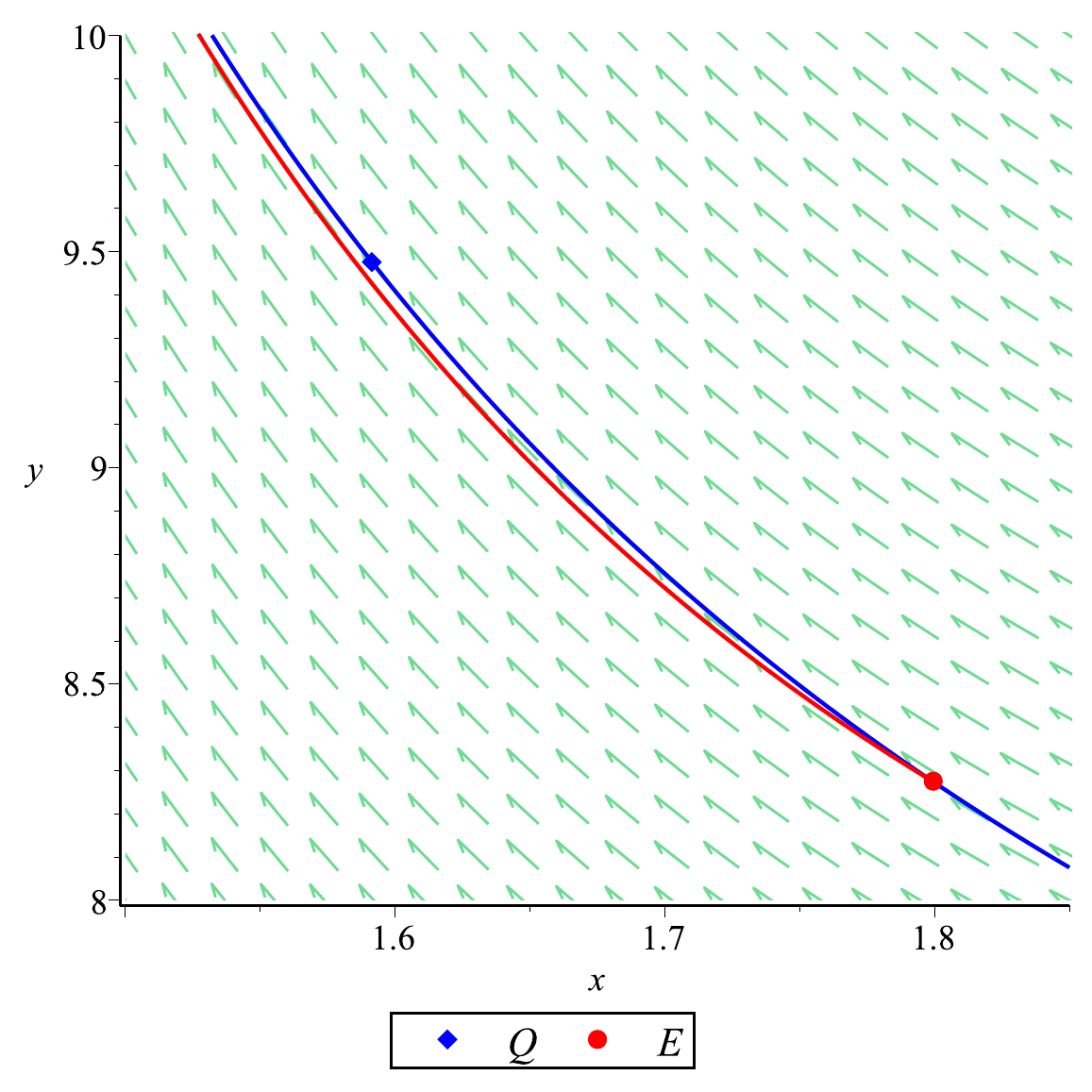}
\caption{Part II: An incoming trajectory originated at $E(1.8, 8.27)$}
\label{improv752}
\end{figure}

\begin{figure}[h!]
\centering
\includegraphics[width=0.5\linewidth]{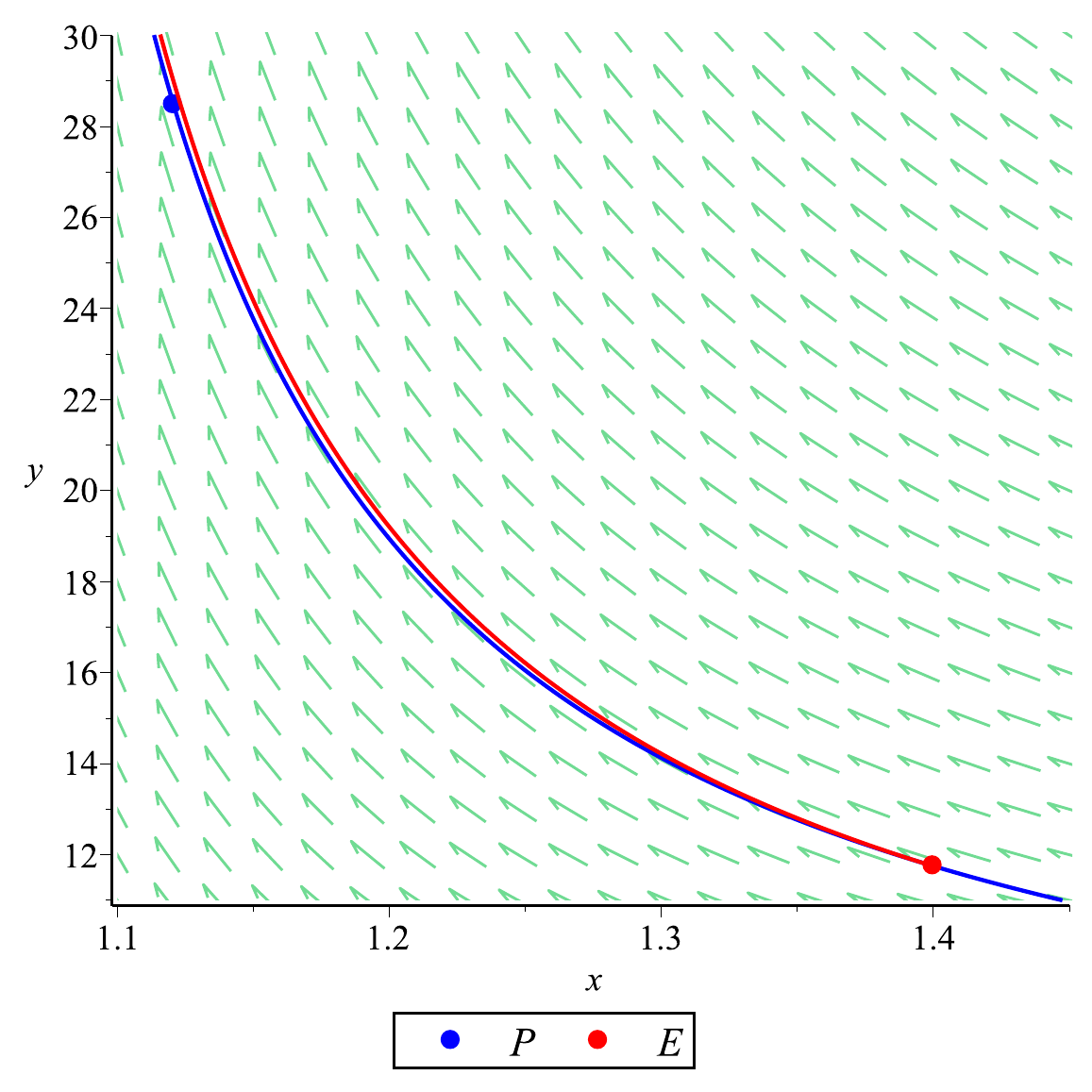}
\caption{Part III: An outgoing trajectory originated at $E(1.4, 11.75)$}
\label{improv76}
\end{figure}

\begin{figure}[h!]
\centering
\includegraphics[width=0.5\linewidth]{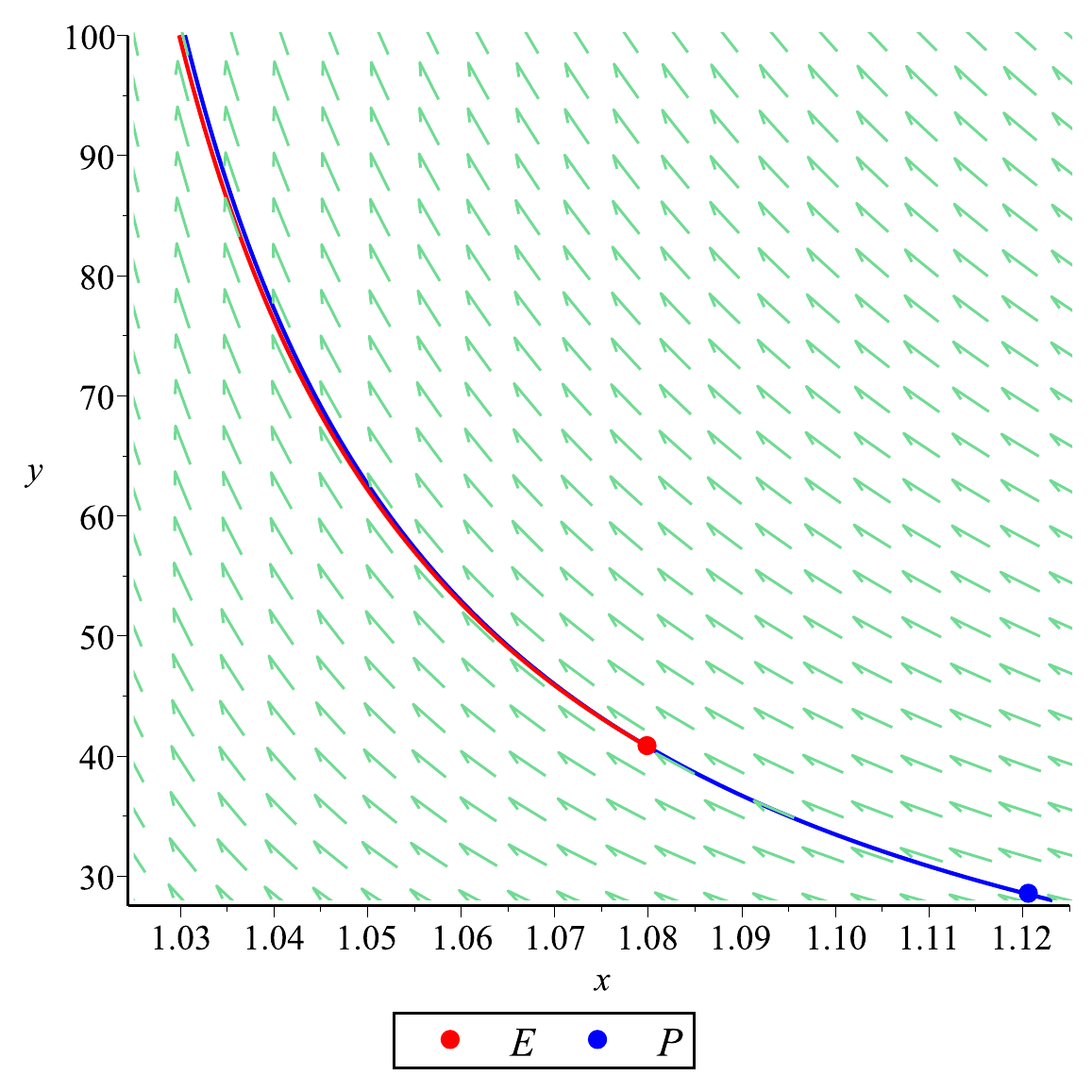}
\caption{Part IV: An incoming trajectory originated at $E(1.08, 40.76)$}
\label{improv762}
\end{figure}

\begin{example}\textit{The case $a\in \big(1/6, a^{\ast}\big)$}.
Not every  metric~\eqref{metric}  preserves  $\operatorname{Ric}>0$
on the spaces $\operatorname{SO}(3k)/\operatorname{SO}(k)\times \operatorname{SO}(k)
\times \operatorname{SO}(k)$ with $k\ge  17$.
\end{example}

As follows from Example~\ref{ex3} the spaces $\operatorname{SO}(3k)/\operatorname{SO}(k)\times \operatorname{SO}(k)
\times \operatorname{SO}(k)$  satisfy  $a\in (1/6, a^{\ast})$ for all $k\ge  17$.
Theorem~7 in~\cite{Ab24} implies also that there exist exactly two points, say $P$  and~$Q$,  on every border curve of the domain $R$ in the case $k\ge 17$ ($\theta<\theta_i$ for each $i$), such that trajectories are leaving $R$ across the arc $PQ$,
whereas every point $E\in (\gamma \setminus PQ)$ emits a trajectory towards~$R$.
Such a  scenario is demonstrated in Figures~\ref{improv75}~--- \ref{improv762} for the case $k=25$ with the corresponding
$$
a=\frac{k}{6k-4}=\frac{25}{146}\approx 0.1712<  a^{\ast}
$$
for various values of the initial point  $E(x_0,y_0)$ chosen in the Cauchy problem.

\section*{Conclusion}

\begin{enumerate}
\item
In~Theorem~\ref{thm1_KarSU_25} we extended some results obtained
in~Theorems~6, 7, 8 of~\cite{Ab24} proving the existence of infinitely many GWS
on which $\operatorname{Ric}>0$ is preserved  for every metric under the NRF~\eqref{ricciflow}.
The number of GWS on which $\operatorname{Ric}>0$ can  be  lost
for every metric is infinite as well.

\item
Theorem~\ref{thm2_KarSU_25} provides some refinement to the results
of Theorem~3 in~\cite{AN} and~Theorem~3 in~\cite{Ab7} devoted to
the special class of GWS with $a_1=a_2=a_3=a$, on which
the conclusion  was obtained using asymptotical methods  that   the NRF~\eqref{ricciflow} evolves any metric~\eqref{metric} to metrics~\eqref{metric}
with  $\operatorname{Ric}>0$   at least at infinity in the case $a\in (1/6, 1/2)$.
Now we know in  detail how this happens:

i) for $a\in [1/6, a^{\ast})$  trajectories of~\eqref{three_equat}
 can both enter and leave the domain~$\mathscr{R}_{+}$
according to the scenario ``enter~-- leave~-- enter''
and after then  never leave $\mathscr{R}_{+}$ again
($\operatorname{Ric}>0$ is preserved for some metrics at least).

ii) for $a\in [a^{\ast}, 1/2)$ trajectories of~\eqref{three_equat}  never leave~$\mathscr{R}_{+}$ originating in~$\mathscr{R}_{+}$
($\operatorname{Ric}>0$ is preserved for all metrics).

\item
Some metrics can lose $\operatorname{Ric}>0$ for $a$ chosen from the narrow interval $[1/6,a^{\ast})$.
But a rigorous analysis shows that those metrics, which lose $\operatorname{Ric}>0$ in the case $a\in(1/6,a^{\ast})$, all restore $\operatorname{Ric}>0$ in a finite time and such a time could be as long as we want. As for $a=1/6$, metrics losing $\operatorname{Ric}>0$ need not restore $\operatorname{Ric}>0$.
Recall that all metrics preserving $\operatorname{Ric}>0$ at $a=1/6$ were found (described) in \cite[Theorem 4]{AN}.

\item
 For further research the question may be of interest to find explicitly all solutions~$a_1,a_2, a_3$
of the system of the inequalities~$\theta\ge \theta_1$, $\theta\ge \theta_2$ and  $\theta\ge \theta_3$, equivalent to
\begin{equation*}
\begin{array}{l}
4(1+2a_1)(a_2+a_3)^2-(1-2a_1)\ge 0,   \\
4(1+2a_2)(a_3+a_1)^2-(1-2a_2)\ge 0,   \\
4(1+2a_3)(a_1+a_2)^2-(1-2a_3)\ge 0.
\end{array}
\end{equation*}
It is easy to observe  that some partial solutions of this system can be given by\,
$a_1=a_2= \epsilon\, (t+1)$,\, $a_3= \dfrac{1-16\epsilon^2}{2(1+16\epsilon^2)}\, (t+1)$\,
for all  $0<\epsilon < \widetilde{a}$ and   all
$0<t< \dfrac{32 \epsilon^2} {1-16\epsilon^2}$.

\item

The another interesting question is that can we obtain a similar detailed analysis in the general case of parameters $a_1,a_2, a_3\in (0,1/2)$. Can we extend the conclusion of Theorem~6 in~\cite{Ab24} to all metrics formulating the following conjecture:
\textit{Every} metric~\eqref{metric} possessing $\operatorname{Ric}>0$ loses $\operatorname{Ric}>0$ on a GWS with $a_1+a_2+a_3\le 1/2$
when evolved by the NRF~\eqref{ricciflow}?

\item
Do we need any additional conditions on the parameters $a_1,a_2, a_3$
in order to the conjecture to be confirmed?
Compare  with Theorem~7 in~\cite{Ab24} stating
``\textit{Every} metric~\eqref{metric} possessing $\operatorname{Ric}>0$ preserves $\operatorname{Ric}>0$ on a GWS with $a_1+a_2+a_3> 1/2$
when evolved by the NRF~\eqref{ricciflow}, if $\theta\ge \theta_i$ for each $i=1,2,3$''.

\end{enumerate}

\vspace{10mm}

\bibliographystyle{amsunsrt}

\end{document}